\documentclass[twoside,11pt]{article}
\usepackage{graphicx,amsmath,latexsym,amssymb,amsthm}
\usepackage[colorlinks=black,urlcolor=black]{hyperref}
\font\chuto=cmbx10 at 16pt \font\kamy=lcmssb8
at 9pt \font\kam=lcmss8 at 8pt 

\date{}
\newtheorem{theorem}{Theorem}[section]
\newtheorem{proposition}[theorem]{Proposition}

\newtheorem{thm}{Theorem}[section]
\newtheorem{lem}[thm]{Lemma}

\newtheorem{prop}[thm]{Proposition}
\newtheorem{rem}[thm]{Remark}

\numberwithin{equation}{section} 

\newcommand{\R}{\ensuremath{\mathbb{R}}}

\newcommand{\dint}{\displaystyle\int}

\begin{document}
%

\centerline {\bf \chuto On jumps stochastic slowly diffusion equations}

\vskip.2cm

\centerline {\bf \chuto   with fast oscillation coefficients  \footnotetext{Email addresses : cmanga@univ-zig.sn (C. Manga), augusteaman5@yahoo.fr (A. Aman), adiedhiou@univ-zig.sn (A. Diédhiou), a.coulibaly5649@zig.sn (A. Coulibaly).}}


\vskip.8cm \centerline {\kamy C. Manga$^\dag,$ A. Aman$^\ddag,$ A. Coulibaly$^\dag$\footnote{{\tt Corresponding author email: a.coulibaly5649@zig.sn} \\
} and A. Diédhiou$^\dag$}

\vskip.5cm

\centerline {$^\dag$Laboratory of Mathematics and Applications, UFR Sciences \& Technologies,}
\centerline {University of Assane SecK, UASZ, BP 523, Ziguinchor, Senegal.} 


\centerline {$^\ddag$Department of Mathematics, UFR de Math\'ematiques et Informatique,}
\centerline {F\'elix H. Boigny University of Cocody, 22 BP 582 Abidjan 22, Ivory Coast.}


\vskip.5cm \hskip-.5cm{\small{\bf Abstract :} We present a large deviation principle for some stochastic evolution equations with jumps which depend on two small parameters, when the viscosity parameter $\varepsilon$ tends to zero more quickly than the homogenization's one $\delta_{\varepsilon}$ (written as a function of $\varepsilon$). In particular, we highlighted a large deviation principle in path-space using some classical  techniques and a uniform upper bound for the characteristic function of a Feller process, in the following sense:
\[
\lim_{\varepsilon\to 0}\dfrac{\delta_{\varepsilon}}{\varepsilon}=+\infty.
\]	

\vskip0.3cm\noindent {\bf Keywords :} Homogenization, Large Deviations, Poisson point process of class.
(QL), Feller semigroup.

\noindent{\bf 2000 Mathematics Subject Classification : } 35B27, 35K57, 60F10, 60G52, 60H15, 60J25.

\hrulefill


\section{Introduction}
\hskip0.6cm
Let $\varepsilon,\delta_{\varepsilon}>0$, we consider a diffusion which allows jumps processes in $\R^d$ satisfying the stochastic differential equations (SDE) :

\begin{equation}
\begin{aligned}\label{sde1}
\begin{split}
X_{t}^{x,\varepsilon,\delta_{\varepsilon}}=x&-\sqrt{\varepsilon}\int_{0}^{t}\sigma\left(\frac{X^{x,\varepsilon,\delta_{\varepsilon}}_{s}}{\delta_{\varepsilon}}\right)dW_{s}+\frac{\varepsilon}{\delta_{\varepsilon}}\int_{0}^{t}b\left(\frac{X^{x,\varepsilon,\delta_{\varepsilon}}_{s}}{\delta_{\varepsilon}}\right)ds\\
&+\int_{0}^{t}c\left(\frac{X^{x,\varepsilon,\delta_{\varepsilon}}_{s}}{\delta_{\varepsilon}}\right)ds+L_{t}^{\varepsilon,\delta_{\varepsilon}},\ x\in\R^{d},
\end{split}
\end{aligned}
\end{equation}
where $\displaystyle\left\lbrace W_{t}:\ t\geqslant 0\right\rbrace $ is a $d$-dimensional standard Brownian motion and $\displaystyle L^{\varepsilon,\delta_{\varepsilon}}:=\left\lbrace L_{t}^{\varepsilon,\delta_{\varepsilon}}:\ t\geqslant 0\right\rbrace $ is a Poisson point process with continuous compensator, independent of $W$, both defined  on a given filtered probability space $\displaystyle\left(\Omega,\mathcal{F},\mathbb{P},\mathbb{F}\right)$ with  $\mathbb{F}:=\displaystyle\left\lbrace\mathcal{F}_{t}:\ t\geqslant 0 \right\rbrace $ being the $\mathbb{P}$-completion of the filtration $\mathcal{F}$.  More precisely, we assume that $L^{\varepsilon,\delta_{\varepsilon}}$ takes the form:
\begin{equation}
L^{\varepsilon,\delta_{\varepsilon}}_{t}:=\int_{0}^{t}\int_{\mathbb{R}^{d}}k\left( \frac{X^{x,\varepsilon,\delta_{\varepsilon}}_{s-}}{\delta_{\varepsilon}},y\right)\left( \varepsilon N^{\varepsilon^{-1}}(dsdy)-\nu(dy)ds\right),\quad t\geqslant 0 
\end{equation}
where $k$ is a given predictable function (see \cite{IW81}, Chap.IV Sect.9), and $N$ is a Poisson random counting measure on $\mathbb{R}^{d}$ with mean L\`{e}vy measure (or intensity measure) $\nu$. Our assumptions on the coefficients $b,c,\sigma$  will be specified below.\\

The main purpose of this paper is to show that under suitable assumptions on the coefficients and the two parameters ($\varepsilon,\delta_{\varepsilon}$), the process $X_{t}^{x,\varepsilon,\delta_{\varepsilon}}$ (\ref{sde1}) satisfies a \textit{large deviation principle} (LDP) with good rate function. The family of equations (\ref{sde1}) subject to $\varepsilon$ (viscosity parameter) and $\delta_{\varepsilon}$ (homogenization parameter) is a classical problem which goes back to Paolo Baldi \cite{Ba91} at the end of 20'th century. In the case when  $L_{t}^{\varepsilon,\delta_{\varepsilon}}\equiv 0$,  the SDE (\ref{sde1}) becomes one driven by an (additive) Brownian motion and the similar issues was  extensively investigated by Freidlin and Sowers \cite{FS99}. They have shown three classical regimes depending on the relative rate at which the small viscosity coefficient $\varepsilon$ and the homogenization parameter $\delta_{\varepsilon}$ tend to zero.  They have provided some effective rate functions associated to an LDP for SDE which have been used as direct applications to wavefront propagation. Subsequently, large deviations problems were studied extensively by many researchers; see, for example, the work of Schilder for LDP of Brownian motions, Sanov for LDP of ergodic processes, and Freidlin and Wentzell for LDP of diffusions. Be that as it may, there are still few results on LDP for stochastic evolution equations with jumps (see, for example \cite{RZ07,Wu11,ZX16}).\\

The main difficulty in the study of SDE (\ref{sde1}), however, is the presence of the jumps. As $\varepsilon$ and $\delta_{\varepsilon}$ tend to zero, two well-known effects come into play. The effect of the viscosity parameter $\varepsilon$ generates small excitations, hence the equation (\ref{sde1}) is called \textit{slowly diffusion}, and the effect of the homogenization parameter $\delta_{\varepsilon}$ makes the coefficients with fast oscillations. In previous publications, we obtained several results on the LDP of equation (\ref{sde1}). In particular we discussed
\begin{itemize}
	\item the case where the homogenization parameter $\delta_{\varepsilon}$ tends to zero much faster than the viscosity's one  $\varepsilon$ (see, \cite{MC119});
	\item the situation in which the two parameters go at the same rate (see, \cite{MC219}).
\end{itemize}	
In this paper, we extends our results to the case that  $\displaystyle\lim_{\varepsilon\to 0}\frac{\delta_{\varepsilon}}{\varepsilon}$ is infinite. That is $\varepsilon$ tends to zero sufficiently quickly compared to $\delta_{\varepsilon}$. To do so, we should first treat $\delta_{\varepsilon}$ as fixed and carry out the calculations for slowly varying $\varepsilon$,  then the theory of \textit{large deviation} tells us how quickly $X^{x,\varepsilon,\delta_{\varepsilon}}$ tends to the deterministic dynamics given by actually setting $\varepsilon$ to zero. \\

Our aim consists in computing the limit of $\varepsilon\log\mathbb{P}\left\lbrace X^{x,\varepsilon,\delta_{\varepsilon}}\in A\right\rbrace $ when $\varepsilon$ and $\delta_{\varepsilon}$ approach zero, where $A$ is a Borel subset of $\mathcal{D}_{x}\left([0,T],\R^d\right) $, the set of c\`{a}dl\`{a}d functions on $[0,T]$ with $\R^d$-values which take $x$ at zero. In other word, we study the LDP for $X^{x,\varepsilon,\delta_{\varepsilon}}$, and prove that there exists  $I_{0,T}: \mathcal{D}_{x}\left([0,T],\R^d\right)\longrightarrow\left[0,+\infty\right] $ such that
\begin{itemize}
	\item for each open set $F\subseteq \mathcal{D}_{x}\left([0,T],\R^d\right)$
	\[
	\liminf_{\varepsilon\to 0}\varepsilon\log\mathbb{P}\left\lbrace X^{x,\varepsilon,\delta_{\varepsilon}}\in F\right\rbrace\geqslant-\inf_{\phi\in F}I_{0,T}(\phi),
	\]
	\item for each closed set $G\subseteq \mathcal{D}_{x}\left([0,T],\R^d\right)$
	\[
	\limsup_{\varepsilon\to 0}\varepsilon\log\mathbb{P}\left\lbrace X^{x,\varepsilon,\delta_{\varepsilon}}\in G\right\rbrace\leqslant-\inf_{\phi\in G}I_{0,T}(\phi).
	\]
\end{itemize}

The rest of this paper is organized as follows. In Section (\ref{s2}) we set up some notation, make precisely our hypothesis
and state the main result. In Section (\ref{s3}) we give the proofs of lower and upper bounds on LDP.

\section{The main results: Preliminaries and formulation}\label{s2}
\subsection{Notation and background}
\noindent Denote expectation with respect to $\mathbb{P}$ by $\mathbb{E}$ and the gradient operator by $\nabla$. We have already defined $\left\langle .,.\right\rangle $  as the standard Euclidean inner product on $\mathbb{R}^{d}$, and $\left\| .\right\| $ as the associated norm. Let $C_{p}\left(\mathbb{R}^{d},\mathbb{R}^{d}\right) $ be the collection of continuous mapping from $\mathbb{R}^{d}$ into $\mathbb{R}^{d}$ which are periodic of period $1$ in each coordinate of the argument and let $\left\| .\right\|_{C_{p}\left(\mathbb{R}^{d},\mathbb{R}^{d} \right) } $ be the associated sup norm. Let $\displaystyle\mathcal{D}\left([0,T],\R^{d}\right) $ be the space of functions that map $[0,T]$ into $\R^{d}$,  which are right continuous and having left hand limits. $\mathcal{D}\left([0,T],\R^{d}\right)$ is metricated by the Skorohod metric, with respect to which it is complete and separable (for a definition and properties of this space see Chap.$3$, \cite{EK86})). We will write $C_{c}^{\infty}\left( \R^d\right) $ for space of test functions. \\

\noindent The $\left\lbrace \sigma_{i}:\ 1\leqslant i\leqslant d \right\rbrace $ in (\ref{sde1}) are assumed to be in $C_{p}\left(\mathbb{R}^{d},\mathbb{R}^{d}\right)$, and we also assume that
\begin{equation}\label{uniel}
\kappa:=\inf\left\lbrace\sum_{i=1}^{d}\left\langle \theta,\sigma_{i}(x)\right\rangle^{2}:\ x\in\mathbb{R}^{d}, \theta\in\mathbb{R}^{d}, \left\| \theta\right\|=1  \right\rbrace>0. 
\end{equation}
We assume that $b,c$ in (\ref{sde1}) are in $C_{p}\left(\mathbb{R}^{d},\mathbb{R}^{d}\right)$. \\

\noindent We now turn our attention to the Poisson part. We first consider a Poisson random measure $N^{\varepsilon^{-1}}(.,.)$ on $\left[ 0,T\right) \times\mathbb{R}^{d} $ defined on the space probability $\left(\Omega,\mathcal{F},\mathbb{P}\right) $,  with L\`{e}vy measure $\varepsilon^{-1}\nu$  such that the standard integrability condition holds:
\begin{equation}\label{spp}
\int_{\mathbb{R}^{d}\backslash\{0\}}\left(1\wedge|y|^{2}\right)\nu(dy)<+\infty. 
\end{equation}
The compensator of $\varepsilon N^{\varepsilon^{-1}}$ is thus the deterministic measure $\displaystyle\varepsilon\tilde{N}^{\varepsilon^{-1}}(dtdy):=dt\nu(dy)$ on $\left[ 0,T\right) \times\mathbb{R}^{d} $ .  In
this paper we shall be interested in \textit{Poisson point process} of class (QL), namely a point process whose counting measure has continuous compensator (see Ikeda and Watanabe \cite{IW81}). More precisely, in light of the representation theorem of the Poisson point process (\cite{IW81}, Chap. II, Theorem $7.4$), we shall assume that $L^{\varepsilon,\delta_{\varepsilon}}$ is a pure jump process of the following form :
\[
L^{\varepsilon,\delta_{\varepsilon}}_{t}:=\int_{0}^{t}\int_{\mathbb{R}^{d}\backslash\{0\}}k\left(\frac{.}{\delta_{\varepsilon}},y\right)(s)\left( \varepsilon N^{\varepsilon^{-1}}(dsdy)-\nu(dy)ds\right),\quad t\geqslant 0,
\]
where $k$ is $C_{p}\left(\mathbb{R}^{d}\times\mathbb{R}^{d},\mathbb{R}^{d}\right)$ with respect to first variable, integrable with respect to $dtdy$,  so that  $L^{\varepsilon,\delta_{\varepsilon}}$ has continuous statistic.\\

\noindent The Markov processes $X^{\varepsilon,\delta_{\varepsilon}}$ that we consider include jump processes and diffusion. Next let's write down its generator on twice continuously differentiable functions with compact support by
\begin{equation}
\begin{split}\label{op}
&\mathcal{L}_{\varepsilon,\delta_{\varepsilon}}\phi\left(x\right):=\\
&\frac{\varepsilon}{2}\sum_{i,j=1}^{d}a_{ij}\left(\frac{x}{\delta_{\varepsilon}}\right)\frac{\partial^{2}\phi(x)}{\partial x_{i}\partial x_{j}} +\frac{\varepsilon}{\delta_{\varepsilon}}\sum_{i=1}^{d}b_{i}\left(\frac{x}{\delta_{\varepsilon}}\right)\frac{\partial\phi(x)}{\partial x_{i}}+\sum_{i=1}^{d}c_{i}\left(\frac{x}{\delta_{\varepsilon}}\right) \frac{\partial\phi(x)}{\partial x_{i}}\\
&+\frac{1}{\varepsilon}\int_{\mathbb{R}^{d}}\left[\phi\left(x+\varepsilon k\left(\frac{x}{\delta_{\varepsilon}},y\right) \right)-\phi(x)-\varepsilon\sum_{i=1}^{d} k_{i}\left(\frac{x}{\delta_{\varepsilon}},y\right)\frac{\partial\phi(x)}{\partial x_{i}}  \right]\nu(dy),
\end{split}
\end{equation}
where the matrix $a:=\left(a_{ij}\right) $ is factored as $a:=\sigma\sigma^{*}$, and $^{*}$ denotes the transpose.  The following hypotheses are required :
\begin{description}
	\item[\textbf{H.1}] (Main hypothesis) $\displaystyle \lim_{\varepsilon\to 0}\frac{\delta_{\varepsilon}}{\varepsilon}=+\infty$.
	\item[\textbf{H.2}] $\left\lbrace 
	\begin{aligned}
	& \textrm{There exists $C_{1}>0$ such that for any $\zeta:=\sigma_{i},b,c,\ 1\leqslant i\leqslant d,\ $ and $\ k$ :}\\
	&\textrm{i)}\left\|\zeta(x')-\zeta(x) \right\|+\int_{\mathbb{R}^{d}}\left\| k(x',y)-k(x,y)\right\|\nu(dy) \leqslant C_{1}\left\|x'-x \right\|,\forall x',x\in\mathbb{R}^{d}. \\
	& \textrm{There exists $C_{2}>0$ such that for any  $\zeta:=\sigma_{i},b,c,\ 1\leqslant i\leqslant d,\ $ and $\ k$ :}\\
	&\textrm{ii)}\left\| \zeta(x)\right\|^{2} + \int_{\mathbb{R}^{d}}\left\| k(x,y)\right\|^{2}\nu(dy) \leqslant C_{2}\left(1+\left\|x\right\|^{2}\right),\quad\forall x\in\mathbb{R}^{d} .
	\end{aligned}
	\right. $
\end{description}
\noindent The proof of the Proposition \ref{prop1} (below) uses the following \textit{Girsanov's formula}. Before proceeding, let us introduce some space. \\
Let $\mathcal{H}^2(T,\lambda)$ be the linear space of all equivalence classes of mappings $F:[0, T] \times \R^d\times \Omega \longrightarrow \mathbb{R}$ 
which coincide almost everywhere with respect to $dt \otimes d\lambda \otimes d\mathbb{P}$  and which satisfy the following conditions :
\begin{itemize}
	\item $ F$ is predictable;
	\item $ \displaystyle \int_0^T \int_{\R^d\backslash\{0\}}\mathbb{E}\Big( | F(t,z)|^2\Big) dt \lambda(dz) < + \infty .$
\end{itemize}
We endow $\mathcal{H}^2(T,\lambda)$ with the inner product 
$$\displaystyle  \left\langle F,G \right\rangle_{T, \lambda}:= \int_0^T \int_{\R^d\backslash\{0\}}\mathbb{E}\Big( F(t,z)G(t,z)\Big) dt \lambda(dz) . $$ 
Then, it is well know that $\left(\mathcal{H}^2(T,\lambda);\left\langle .,. \right\rangle_{T, \lambda}\right)$  is a real separable Hilbert space.\\
Let $N_{p}$  be a Poisson random measure on $\mathbb{R}_{+}\times\left(\R^d\backslash\{0\}\right) $ with intensity measure $\lambda$, according to a given $\mathcal{F}_{T}$-adapted, $\sigma$-finite point process $p$ which is independent of the Brownian motion $W$. Let $\widetilde{N}_{p}$ be the associated compensated Poisson random measure. Now we have (see, D. Applebaum \cite{Ap09}  Chapter $5$, Section $2$)
\begin{lem}[\textbf{Girsanov's formula}]$\quad$
	\\ Let $X$ be a L\'evy process such that $e^{X}$ is a martingale, \textit{i.e}:
	\[
	X_t= \int_0^t b(s) ds + \int_0^t \sigma (s) d W_s + \int_0^t \int_EH(s, z)\widetilde{N}_{p}(dsdz)+   \int_0^t \int_E K(s,z)N_{p}(dsdz),
	\] with
	\[
	b(t)= -\frac{1}{2}\sigma^2 (t) - \int_E \left( e^{H(t, z)}-1 - H(t,z)\right) \lambda(dz) - \int_E\left( e^{K(t, z)}-1 \right)\lambda(dz), \  \mathbb{P}-a.s.
	\]
	We suppose that there exists $C>0$ such that
	\[
	\left| K(t, z)\right|\leqslant C,\quad \forall t\geqslant0, \forall z\in E .
	\]
	For $L \in \mathcal{H}^2 (T, \lambda)$ we define
	\[
	M_t := \int_0^t \int_{z\neq 0} L(s, z) \widetilde{N}(dsdz).
	\]
	Set
	\[
	U(t,z)= \left( e^{H(t, z)}-1 \right)\boldsymbol{1}_{\big\{\left\|z\right\|< 1 \big\} }+ \left( e^{K(t, z)}-1 \right)\boldsymbol{1}_{\big\{\left\|z\right\|\geqslant 1 \big\} }
	\] and we suppose that
	\[
	\int_0^T \int_{\big\{\left\|z\right\|\leqslant 1 \big\}} \left( e^{H(s, z)}-1 \right)^2 \lambda(dz)ds < + \infty .
	\]
	Finally, we define 
	\[
	B_t= W_t - \int_0^t \sigma(s) ds\quad
	\textrm{ and }\quad
	N_t= M_t -\int_0^t \int_{z\neq 0} L(s, z)U(s,z)\lambda(dz)ds, \ \ 0 \leq t \leq T.
	\]
	Let $\boldsymbol{Q}$ be the probability measure on $(\Omega, \mathcal{F}_T)$ defined as:
	\[
	\dfrac{d\boldsymbol{Q}}{d\mathbb{P}}:= e^{\displaystyle X_T}.
	\]
	Then under $\boldsymbol{Q}$, $B_t $ is a Brownian motion and $N_t$ is a $\boldsymbol{Q}$-martingale.
\end{lem}
The following remark will be useful in the proof of Proposition \ref{prop1}.
\begin{rem}
	It is well known that if $N$ is a Poisson process with intensity $\lambda(s)$ and with compensated martingale associated  $M$, and  If $g$ is a $\left] -1,+\infty\right[ $-values bounded Borel function   then the following is a martingale,
	\[
	X_{t}:=\exp\left( \int_{0}^{t}\log\left(1+g(s) \right)dM_{s}-\int_{0}^{t}\left[ \log\left(1+g(s) \right)-g(s)\right]\lambda(s)  \right) .
	\] 
\end{rem}

\noindent The proof of the Proposition \ref{prop2} (below) uses an \textit{analytic} approach and appeals to the  \textit{classical Hartman-Wintner condition} (see, Lemma  \ref{hwc}). \\

\noindent In order to simplify the further exposition, we briefly outline the method which we use to bound the density
function $p(t, x, y)$ with respect to Lebesgue measure. Let $\left( \boldsymbol{T}_{t}\right)_{t\geqslant 0} $ be a semigroup on $C^{\infty}\left(\R^d\right) $ the space of continuous functions vanishing at infinity. We define its (infinitesimal) generator $\mathcal{L}$ as
follows:
\[
\mathcal{L}\phi:=\lim_{t\to 0}\dfrac{\boldsymbol{T}_{t}\phi -\phi}{t}, \textrm{ for all }\phi\in \mathcal{D}(\mathcal{L}).
\]
Here $\mathcal{D}(\mathcal{L})$ is the set of all $\phi$ on $C^{\infty}\left(\R^d\right) $ for which the limit exists in a strong sense,\textit{ i.e.}. with respect to the \textit{sup-norm}. A useful approach to the study of the generator in case it is a pseudo-differential operator is to study its \textit{symbol}. Let $\mathcal{L}$ be a generator of a Feller process with $C_{c}^{\infty}\left( \R^d\right)\subset \mathcal{D}\left( \mathcal{L}\right)$.  Then the
restriction of $\mathcal{L}$ on $C_{c}^{\infty}\left( \R^d\right)$ is a pseudo-differential operator, given by
\begin{equation}
\mathcal{L}\phi(x) :=-\dint e^{i\left\langle x,\xi\right\rangle }q(x,\xi)\hat{\phi}(\xi)d\xi\ \textrm{ for }\phi\in C_{c}^{\infty}\left( \R^d\right),
\end{equation}
where $\hat{\phi}$ denotes the Fourier transform of $\phi$ described by 
$$\displaystyle\hat{\phi}(\xi)=\frac{1}{(2\pi)^{d/2}}\int e^{-i\left\langle x,\xi\right\rangle }\phi(x)dx,\quad \xi\in\R^d.$$
The function $q(.,.) : \R^d\times\R^d\longrightarrow \mathbb{C} $ is called the symbol of the operator $\mathcal{L}$. The symbol $q(x,\xi)$ is locally bounded in $(x,\xi)$, measurable as a function of $x$, and for every fixed $x\in\R^d$ it is a continuous negative definite function in the co-variable. That is, it is given by a L\`{e}vy-Khintchine formula of the form 
\begin{equation}
\begin{split}
q(x,\xi):=c(x)&-i\left\langle d(x),\xi\right\rangle+\frac{1}{2}\left\langle \xi,Q(x)\xi\right\rangle\\
&+\int_{z\neq 0} \left( 1-e^{i\left\langle z,\xi\right\rangle }+i\left\langle z,\xi\right\rangle \boldsymbol{1}_{\left\lbrace \left\| z\right\| \leqslant 1\right\rbrace }\right) \lambda\left( x,dz\right),
\end{split}
\end{equation} 
with $c\geqslant 0,\  d\in\R^{d},\  Q\in\R^{d\times d}$ a positive-semidefinite symmetric matrix and $\lambda$ a is a non-negative, $\sigma$-finite kernel on $\R^{d}\times \mathcal{B}\left( \R^{d}\backslash \{0\}\right) $. A detailed exposition of the use of the symbol in the study of Markov processes can be found in \cite{Ja01,Ja02,Ja05}. It is well know  that Feller generators are \textit{variable coefficient} L\`{e}vy-type operators
and once we fix $x$ then $ -p$ is the generator of a L\`{e}vy process. Notice that $p$ is no longer the characteristic exponent of the Feller
process $\left( Y_{t}\right)_{t\geqslant 0} $, that is, the formula $\displaystyle \mathbb{E}\left(e^{\displaystyle i\left\langle \xi,Y_{t}-x\right\rangle }\right) = e^{\displaystyle -tq(x,\xi)}$ is, in general, an inaccurate result. However, it is natural to expect that
\[
\mathbb{E}\left(e^{\displaystyle i\left\langle \xi,Y_{t}-x\right\rangle }\right) \approx e^{\displaystyle -tq(x,\xi)}.
\]
The following plays a pivotal role in the sequel, because it allows us to show the existence of a transition density of a Feller process and to link it explicitly in terms of symbol (see, for instance \cite{SW19}).
\begin{lem}[\bf Existence of density]$\quad$\label{hwc}\\
	Let $\left( \left( Z_{t},\right)_{t\geqslant 0}, \mathbb{P}\right) $ be a Feller process with generator $\left( \mathcal{L},\mathcal{D}\left( \mathcal{L}\right) \right) $, such that $C_{c}^{\infty}\left( \R^d\right)\subset \mathcal{D}\left( \mathcal{L}\right)$. Then $\mathcal{L}\arrowvert_{C_{c}^{\infty}\left( \R^d\right)}=-q\left( .,D\right) $ is a pseudo-differential operator with symbol $q$. Assume
	that $q$ satisfies the properties: for some $C > 0 $ and for all $\xi\in\R^d$
	\begin{equation}\label{c1}
	\left\| q(.,\xi)\right\|_{\infty} \leqslant C\left( 1+\left\| \xi\right\| ^2\right)\  \textrm{ and }\ q(.,0)=0.
	\end{equation}
	If moreover (\textnormal{\bf Hartman-Wintner condition})
	\begin{equation}\label{c2}
	\lim_{\left\| \xi \right\| \to \infty}\dfrac{\inf_{z\in\R^d}\textrm{\textnormal{Re} }q\left( z,\xi\right) }{\log\left(1+\left\|\xi \right\|  \right) }=\infty,
	\end{equation}
	then the process $\left( Z_{t}\right)_{t\geqslant 0}$ has a transition density $p(t, x, y),\ t\in\left[0,\infty\right),\   x, y\in\R^d$, with respect
	to the Lebesgue measure and the following inequality holds for $t > 0$:
	\begin{equation}
	\sup_{x,y\in\R^d}p(t,x,y)\leqslant\dint\exp\left( -\frac{t}{16}\inf_{z\in\R^d}\textrm{\textnormal{Re} }q\left( z,\xi\right) \right) d\xi.
	\end{equation}
\end{lem}

\noindent Before finishing this section, we point out that if the canonical process is Feller under $\mathbb{P}$, so that all  requirements of Lemma \ref{hwc} are satisfied, it is a \textit{strong} Feller process, that is, its semigroup maps  bounded measurable functions to continuous bounded functions. In this case, the symbol can also be written as
\begin{equation}
q(x,\xi):=-\lim_{t\to 0}\frac{\mathbb{E}\left( e^{i\left\langle Y_{t}-x,\xi\right\rangle }\right) -1}{t}.
\end{equation}
Hence, the symbol can be probabilistically  interpreted as the derivative of the characteristic function of the process  (defined entirely in analytic terms), \textit{i.e.}
\begin{equation}
\frac{\textrm{d}}{\textrm{dt}}\lambda_{t}(x,\xi)\Big\arrowvert_{t=0}=-q(x,\xi)=e^{-i\left\langle x,\xi\right\rangle }\mathcal{L}e^{i\left\langle x,\xi\right\rangle },\quad x,\xi\in\R^d,
\end{equation}
where $\displaystyle\lambda_{t}(x,\xi):=e^{-i\left\langle x,\xi\right\rangle }\boldsymbol{T}_{t}e^{i\left\langle x,\xi\right\rangle }(x)$.
\subsection{The main results}
Before proceeding, let us have some definitions.\\

\noindent $(\boldsymbol{D.1})\ V_{L}: \R^d\times\R^d\times\mathcal{H}^{2}\left( L,\nu\right) \longrightarrow\left[ 0,+\infty\right]$  the energy function defined as: 
\begin{equation*}\label{enerf}
V_{L}(x,z,\phi):=V_{L}^{1}(x,z)+V_{L}^{2}(\phi)
\end{equation*}  
where
\[
\begin{aligned}
&V^{1}_{L}(x,z):=\inf_{\substack{\psi\in \mathcal{D}\left(\left[0,L\right];\ \R^d\right)\\\psi_0=x,\ \psi_L=z}}\frac{1}{ 2}\int_{0}^{L}\left\|\dot{\psi}_s-c\left(\psi_s\right) -k\left( \psi_{s}\right) \right\|^{2}_{a^{-1}\left(\psi_s\right)} ds\\
&V^{2}_{L}(\phi):=\inf_{\substack{\phi\in \mathcal{H}^{2}\left(L,\nu\right)\\\phi\geqslant 0}}\int_{0}^{L}\int_{\R^d}\Big(\phi \log\left(\phi\right) -\phi+1\Big)(s,y) ds\nu(dy),
\end{aligned}
\]
\noindent with the norm $\left\|\theta \right\|_{a^{-1}}:=\sqrt{\left\langle\theta,a^{-1}\theta\right\rangle}$ for all $\theta\in\mathbb{R}^d$, and with $\ \displaystyle k(z):=\int_{\R^d}k(z,y)\nu(dy)$.\\

\noindent$(\boldsymbol{D.2})\ \mathcal{J} : \R^{2}\rightarrow\left[0\right.,\left.\infty\right) $  the functional given by:
\begin{equation*}\label{ftaux}
\mathcal{J}(z)=\lim_{L\to +\infty}\frac{\scriptstyle 1}{\scriptstyle L}V_{L}\left( 0,L z, 1+\phi\right).
\end{equation*} 

\noindent The following is our main result.
\begin{thm}\label{main theorem}
	Fix $x\in\R^d$, assume (\textbf{H.1}) and (\textbf{H.2}) hold. Then we have	
	\begin{itemize}
		\item for each open set $G\subseteq \mathcal{D}\left([0,T],\R^d\right)$
		\[
		\liminf_{\varepsilon\to 0}\varepsilon\log\mathbb{P}\left\lbrace X_{T}^{x,\varepsilon,\delta_{\varepsilon}}\in G\right\rbrace\geqslant-\inf_{\substack{\varphi\in \mathcal{D}\left([0,T],\R^{d}\right)\\
				\varphi(0)=x,\ \varphi(T)=z}}\int_{0}^{T}\mathcal{J}\left(\dot{\varphi}(s)\right)ds,
		\]
		\item for each closed set $F\subseteq \mathcal{D}\left([0,T],\R^d\right)$
		\[
		\limsup_{\varepsilon\to 0}\varepsilon\log\mathbb{P}\left\lbrace X_{T}^{x,\varepsilon,\delta_{\varepsilon}}\in F\right\rbrace\leqslant-\inf_{\substack{\varphi\in \mathcal{D}\left([0,T],\R^{d}\right)\\
				\varphi(0)=x,\ \varphi(T)=z}}\int_{0}^{T}\mathcal{J}\left(\dot{\varphi}(s)\right)ds.
		\]
	\end{itemize}
\end{thm}

\section{Large Deviation Principle}\label{s3}
Before proceeding, we observe that the function $\mathcal{J}$ is convex, hence we can show that
\[
\inf_{\substack{\varphi\in \mathcal{D}\left([0,T],\R^{d}\right)\\
		\varphi(0)=x,\ \varphi(T)=z}}\int_{0}^{T}\mathcal{J}\left(\dot{\varphi}(s)\right)ds:=T\mathcal{J}\left(\dfrac{z-x}{T} \right) .
\]
Next we are going to give the outline of the proof.
\subsection{The lower bound}
We start with the following \textit{lower bound} in space.
\begin{prop}\label{prop1}
	Suppose the assumptions ($ \boldsymbol{H.1}$) to ($\boldsymbol{H.2}$) hold. For each open subset $\ G\subseteq \R^{d}$ we have
	\[
	\liminf_{\varepsilon\to 0}\varepsilon\log\mathbb{P}\left\lbrace X_{1}^{x,\varepsilon,\delta_{\varepsilon}}\in G\right\rbrace \geqslant -\inf_{ z\in G } \mathcal{J}(z-x).
	\]
\end{prop}
\begin{proof}
	Let $\psi\in \mathcal{D}\left([0,1],\R^{d}\right)  $ satisfying $\psi_0=x$, $\psi_1=z$. Let us set $$\hat{X}_{t}^{\varepsilon,\delta_{\varepsilon}}:=\dfrac{1}{\sqrt{\varepsilon}}\left(X_{t}^{x,\varepsilon,\delta_{\varepsilon}}-\psi(t)\right).$$
	
	\noindent Now, fix $z\in G$, $\eta>0$ and let $\delta_{\varepsilon}'>0$ be small enough  so that 
	\[
	\left\lbrace z'\in\R^{d} : \left\|z'-z \right\|\leqslant \eta\delta_{\varepsilon}'  \right\rbrace \subseteq G.
	\]
	Fix also $0<\delta_{\varepsilon}<\delta_{\varepsilon}'$, then
	\[
	\begin{split}
	\mathbb{P}\left\lbrace X_{1}^{x,\varepsilon,\delta_{\varepsilon}}\in G\right\rbrace &\geqslant \mathbb{P}\left\lbrace\left\|X_{1}^{x,\varepsilon\delta_{\varepsilon}}-z \right\|\leqslant \eta\delta_{\varepsilon}\  \right\rbrace\\
	&\geqslant \mathbb{P}\left\lbrace\left\|X_{1}^{x,\varepsilon\delta_{\varepsilon}}-\psi \right\|_{\mathcal{D}\left([0,1],\R^{d}\right)}\leqslant \eta\delta_{\varepsilon}  \right\rbrace\\
	& \geqslant \mathbb{P}\Big\{\underbrace{\big\Vert\hat{X}_{1}^{\varepsilon,\delta_{\varepsilon}}\big\Vert_{\mathcal{D}\left([0,1],\R^{d}\right)}\leqslant\frac{\delta_{\varepsilon}}{\sqrt{\varepsilon}}\eta  }_{\displaystyle:=A_{\varepsilon,\delta_{\varepsilon}}^{\eta}}\Big\}.
	\end{split} 
	\]
	
	\noindent Next, define
	\[
	\begin{aligned}
	&\begin{split}
	\xi(t):=&\left[ \dot{\psi}(t) -\frac{\varepsilon}{\delta_{\varepsilon}}b\left(\frac{\scriptstyle X_{t-}^{x,\varepsilon,\delta_{\varepsilon}}}{\scriptstyle\delta_{\varepsilon}}\right) -c\left(\frac{\scriptstyle X_{t-}^{x,\varepsilon,\delta_{\varepsilon}}}{\scriptstyle\delta_{\varepsilon}}\right)\right. \\
	&\qquad\left. +\int_{\R^d}k\left(\frac{\scriptstyle X_{t-}^{x,\varepsilon,\delta_{\varepsilon}}}{\scriptstyle\delta_{\varepsilon}},y\right)\nu(dy)\right] \times\sigma^{-1} \left(\frac{X_{t}^{x,\varepsilon,\delta_{\varepsilon}}}{\delta_{\varepsilon}}\right),
	\end{split}\\
	&\hat{W}_{t}:=W_{t}-\frac{1}{\sqrt{\varepsilon}}\int_{0}^{t}\xi(s)ds,\\
	&\begin{split}
	\hat{N}_{z}^{\varepsilon^{-1}}\left([0,t],U\right) :=&\frac{\scriptstyle1}{\scriptstyle\varepsilon}\log\left(1+\phi(t,y)\right)\left( \varepsilon N^{\varepsilon^{-1}}\left([0,t],U\right)- \nu(U)\right) \\
	&-\frac{\scriptstyle1}{\scriptstyle\varepsilon}\log\left(1+\phi(t,y)\right)\Bigg(\phi(s,y)\boldsymbol{1}_{\{\left\| y\right\| < 1\}}  \\
	&\qquad\qquad\qquad\qquad\qquad+\left(e^{k(z,y)}-1\right)\boldsymbol{1}_{\{\left\| y\right\| \geqslant 1\}}\Bigg)\nu(U),
	\end{split}
	\end{aligned}
	\]
	for $z\in\R^d,\ U\in\mathcal{B}\left(\R^{d}\right) $ and $\phi\in\mathcal{H}^{2}(1,\nu)$.\\
	
	\noindent By Girsanov's formula, we introduce the measure $\hat{\mathbb{P}}$ on $\left(\Omega,\mathcal{F}\right)$ defined as :
	\[
	\begin{split}
	&\dfrac{d\hat{\mathbb{P}}}{d\mathbb{P}}:=e^{\left(-\dfrac{\scriptstyle1}{\scriptstyle2\varepsilon}\dint_{0}^{1}\left\|\xi(s) \right\|^{2}ds-\frac{\scriptstyle 1}{\scriptstyle\varepsilon}\int_{0}^{1}\int_{\R^d} \Big(\phi(s,y)-\log\left(1+\phi(s,y)\right) \Big)\nu(dy)ds\right)  }\\
	&\times e^{\left(\displaystyle\frac{\scriptstyle1}{\scriptstyle\sqrt{\varepsilon}}\int_{0}^{1}\xi(s)dW_{s}-\frac{\scriptstyle 1}{\scriptstyle\varepsilon}\int_{0}^{1}\int_{\R^d}\log\left( 1+\phi(s,y)\right)\left(\varepsilon N^{\varepsilon^{-1}}(dsdy)-\nu(dy)ds\right)\right) }\\
	&\times e^{\left(\displaystyle -\frac{\scriptstyle 1}{\scriptstyle\varepsilon}\int_{0}^{1}\int_{\R^d}k\left( \frac{\scriptstyle X^{x,\varepsilon,\delta_{\varepsilon}}_{s-}}{\scriptstyle\delta_{\varepsilon}},y\right)\varepsilon N^{\varepsilon^{-1}\log\left( 1+\phi(s,y)\right)}(dsdy)\right) }\\
	&\times e^{\Bigg(\displaystyle\frac{\scriptstyle 1}{\scriptstyle\varepsilon}\int_{0}^{1}\int_{\R^d}\log\left(1+\phi(s,y)\right)\Bigg(e^{\displaystyle k\left( \frac{\scriptstyle X^{x,\varepsilon,\delta_{\varepsilon}}_{s-}}{\scriptstyle\delta_{\varepsilon}},y\right)}-1\Bigg)\boldsymbol{1}_{\{\left\| y\right\| \geqslant 1\}}\nu(dy)ds\Bigg)}.
	\end{split}
	\]
	It follows that
	\begin{equation}\label{girsa}
	\begin{split}
	&\mathbb{P}\left( A_{\varepsilon,\delta_{\varepsilon}}^{\eta}\right)=\hat{\mathbb{E}}\Bigg\{ \boldsymbol{1}_{\displaystyle A_{\varepsilon,\delta_{\varepsilon}}^{\eta}}e^{\displaystyle\left(  \frac{\scriptstyle1}{\scriptstyle\sqrt{\varepsilon}}\int_{0}^{1}\xi(s)dW_{s}- \int_{0}^{1}\int\hat{N}_{X/\delta_{\varepsilon}}^{\varepsilon^{-1}}(dsdy)		
		-M_{1}^{\varepsilon,\phi}\right)} \\
	&\times\boldsymbol{1}_{\displaystyle A_{\varepsilon,\delta_{\varepsilon}}^{\eta}} e^{\displaystyle\left(-\frac{\scriptstyle1}{\scriptstyle2\varepsilon}\int_{0}^{1}\left\|\xi(s) \right\|^{2}ds 	  -\frac{\scriptstyle1}{\scriptstyle\varepsilon}\int_{0}^{1}\int_{\R^d}\hat{\Phi}(s,y,\phi)\nu(dy)ds\right)}\Bigg\}.
	\end{split}
	\end{equation}
	where
	\[
	\begin{aligned}
	&\hat{M}^{\varepsilon, \phi }_{t}:=\frac{\scriptstyle 1}{\scriptstyle\varepsilon}\int_{0}^{t}\int_{\R^d}k\left( \frac{\scriptstyle X^{x,\varepsilon,\delta_{\varepsilon}}_{s-}}{\scriptstyle\delta_{\varepsilon}},y\right)\varepsilon N^{\varepsilon^{-1}\log\left( 1+\phi(s,y)\right)}(dsdy),\\
	&\hat{\Phi}\left(t,y,\phi\right):=\big(1+\phi(t,y)\big)\log\left( 1+\phi(t,y)\right)- \phi(t,y).
	\end{aligned}
	\]
	Notice that
	\begin{equation*}
	A_{\varepsilon,\delta_{\varepsilon}}^{\eta}\equiv\left\lbrace\sup_{0\leqslant t\leqslant 1}\Bigg\Vert\frac{X_{t}^{x,\varepsilon,\delta_{\varepsilon}}}{\delta_{\varepsilon}}-\frac{\psi(t)}{\delta_{\varepsilon}}\Bigg\Vert\leqslant\eta \right\rbrace
	\end{equation*}
	and on this set, we have 
	$$
	\displaystyle \frac{\scriptstyle1}{\scriptstyle2}\int_{0}^{1}\left\|\xi(s) \right\|^{2}ds\leqslant \hat{V}^{\varepsilon,\delta_{\varepsilon}}(1;\psi,\phi)
	$$
	where
	\begin{equation}\label{argint}
	\begin{split}
	\hat{V}_{\eta}^{\varepsilon,\delta_{\varepsilon}}(t,\psi,\phi):=&
	\sup_{\left\lbrace \Vert\varrho\Vert_{ \mathcal{D}\left(\left[0,t\right];\R^d\right)}\leqslant \eta \right\rbrace}\frac{\scriptstyle1}{\scriptstyle 2}\int_{0}^{t}\left\|\dot{\psi}_{s}-B^{\varepsilon,\delta_{\varepsilon}}\left(\frac{\psi_s}{\delta_{\varepsilon}}+\varrho(s)\right)\right\|_{a^{-1}\left(\frac{\psi_s}{\delta_{\varepsilon}}+\varrho(s)\right)}^{2}ds
	\end{split}
	\end{equation}
	and with 
	\[
	B^{\varepsilon,\delta_{\varepsilon}}(z):=\frac{\varepsilon}{\delta_{\varepsilon}}b(z)+c(z)-\int_{\R^d}k(z,y)\nu(dy).
	\]
	For any $\varrho\in \mathcal{D}\left(\left[0,1\right];\R^d\right)$ with $\Vert\varrho\Vert_{ \mathcal{D}\left(\left[0,1\right];\R^d\right)}\leqslant \eta$, as in \cite{FS99} (Young's inequality), there exists constants $\kappa_{1}$ and $\kappa_{2}>0$  such that for all $\tilde{\eta}>0,$ 
	\begin{equation}\label{inqq}
	\begin{split}
	\frac{1}{ 2}&\left\|\dot{\psi}_{t}-B^{\varepsilon,\delta_{\varepsilon}}\left(\frac{\psi_t}{\delta_{\varepsilon}}+\varrho(t)\right) \right\|_{a^{-1}\left(\frac{\psi_t}{\delta_{\varepsilon}}+\varrho(t)\right)}^{2}\\
	&\qquad\leqslant \frac{1}{2}(1+\kappa_{1}\eta)(1+\tilde{\eta})\left\|\dot{\psi}_{t}-c\left( \frac{\psi(t)}{\delta_{\varepsilon}}\right)+k\left( \frac{\psi(t)}{\delta_{\varepsilon}}\right) \right\|_{a^{-1}\left(\frac{\psi_t}{\delta_{\varepsilon}}\right)}^{2} +\omega_{\varepsilon,\delta_{\varepsilon}}^{\eta,\tilde{\eta}},
	\end{split}
	\end{equation}
	where
	\[
	\begin{split}
	\omega_{\varepsilon,\delta_{\varepsilon}}^{\eta,\tilde{\eta}}:=&\frac{ 1}{ 2}(1+\kappa_{1}\eta)\left(1+\tilde{\eta}^{-1}\right)\kappa_{2}\sup_{\left\lbrace \substack{ z,z'\in\R^d\\\left\|z-z' \right\|\leqslant \eta}\right\rbrace }\left\lbrace  \frac{\varepsilon}{\delta_{\varepsilon}}\left\|b(z)\right\|+\left\| c(z)-c(z')\right\|\right. \\
	& +\int_{\R^d}\left\| k(z,y)-k(z',y) \right\|  \boldsymbol{1}_{\{z\neq 0\}}\nu(dy)\Big\}  . 
	\end{split} 
	\]
	From (\ref{girsa}) we have
	\begin{equation*}
	\mathbb{P}(A_{\varepsilon,\delta_{\varepsilon}}^{\eta})\geqslant\exp{\left(- \dfrac{\Psi_{\varepsilon,\eta}(1,\psi,\phi)}{\varepsilon}\right) }\times\hat{\mathbb{E}}_{\boldsymbol{\displaystyle 1}_{A_{\varepsilon,\delta_{\varepsilon}}^{\eta}}}\left\lbrace \exp\Big(-\frac{\scriptstyle1}{\scriptstyle\sqrt{\varepsilon}}\big|\Lambda_{\varepsilon}(W,N,\hat{N})\big|\Big)\right\rbrace 
	\end{equation*}
	where 
	\[
	\begin{split}
	\Psi_{\varepsilon,\eta}(1,\psi,\phi)&:=\hat{V}_{\eta}^{\varepsilon,\delta_{\varepsilon}}(1,\psi,\phi)+\sup_{\{\phi \in \mathcal{H}^2(1, \nu): \phi >-1\}}\int_{0}^{1}\int_{\R^d}\hat{\Phi}(s,y, \phi)\nu(dy)ds\\
	\Lambda_{\varepsilon}(W,N,\hat{N})&:= \int_{0}^{1}\xi(s) dW_{s}-\sqrt{\varepsilon}\int_{0}^{1}\int_{\R^d} \hat{N}_{X/\delta_{\varepsilon}}^{\varepsilon^{-1}}(dsdy)- \sqrt{\varepsilon} M_1^{\varepsilon, \phi}.
	\end{split}
	\]
	As a remind $\displaystyle \int_{0}^{1}\int_{\R^d} \hat{N}_{X/\delta_{\varepsilon}}^{\varepsilon^{-1}}(dsdy)$ is a maringal under $\hat{\mathbb P}$.
	
	\noindent Girsanov's theorem tells us that $\hat{\mathbb{P}}(A_{\varepsilon,\delta_{\varepsilon}}^{\eta})\longrightarrow 1$ when $\varepsilon,\delta_{\varepsilon}\to 0$. So, for $\varepsilon,\delta_{\varepsilon}>0$  sufficiently small, $\hat{\mathbb{P}}(A_{\varepsilon,\delta_{\varepsilon}}^{\eta})$ is positive. Thus, we have
	\begin{equation}
	\begin{split}
	\mathbb{P}(A_{\varepsilon,\delta_{\varepsilon}}^{\eta})\geqslant & \exp{\left( -\frac{\Psi_{\varepsilon,\eta}(1,\psi,\phi)}{\varepsilon}\right)}\times \hat{\mathbb{P}}(A_{\varepsilon,\delta_{\varepsilon}}^{\eta})\\ &\times\dfrac{\hat{\mathbb{E}}_{\boldsymbol{\displaystyle 1}_{A_{\varepsilon,\delta_{\varepsilon}}^{\eta}}}\left\lbrace \exp-\frac{\scriptstyle1}{\scriptstyle\sqrt{\varepsilon}}\left|\Lambda_{\varepsilon}(W,N,\hat{N})\right|\right\rbrace }{\hat{\mathbb{P}}(A_{\varepsilon,\delta_{\varepsilon}}^{\eta})}.
	\end{split}
	\end{equation}
	Therefore,
	\begin{equation}
	\begin{split}
	\mathbb{P}(A_{\varepsilon,\delta_{\varepsilon}}^{\eta})\geqslant &\exp{\left( -\frac{\Psi_{\varepsilon,\eta}(1,\psi,\phi)}{\varepsilon}\right)}\times\hat{\mathbb{P}}(A_{\varepsilon,\delta_{\varepsilon}}^{\eta})\\
	&\times\underbrace{\exp{\Bigg( -\frac{\scriptstyle1}{\scriptstyle\sqrt{\varepsilon}}\dfrac{\left[\hat{\mathbb{E}}_{\boldsymbol{\displaystyle 1}_{A_{\varepsilon,\delta_{\varepsilon}}^{\eta}}}\left|\Lambda_{\varepsilon}(W,N,\hat{N})\right| \right]}{\hat{\mathbb{P}}(A_{\varepsilon,\delta_{\varepsilon}}^{\eta})}\Bigg)} }_{\textrm{ Jensen's inequality}}
	\end{split}
	\end{equation}
	and
	\begin{equation}
	\begin{split}
	\mathbb{P}(A_{\varepsilon,\delta_{\varepsilon}}^{\eta})\geqslant &  \exp{\Bigg(-\frac{\Psi_{\varepsilon,\eta}(1,\psi,\phi)}{\varepsilon}\Bigg)}\\&
	\times	\underbrace{\exp{\Bigg( -\frac{\scriptstyle K^{1/\kappa_{3}}}{\scriptstyle \sqrt{\varepsilon}}\frac{\sqrt{2\hat{V}^{\varepsilon,\delta_{\varepsilon}}_{\eta}(1,\psi,\phi)+C(1+ \frac{1}{\delta_{\varepsilon} })}}{\left( \tilde{\omega}_{\varepsilon,\delta_{\varepsilon}}\right)^{1/\kappa_{3}} } \Bigg)}\times\tilde{\omega}_{\varepsilon,\delta_{\varepsilon}}}_{\textrm{ Burkholder-Davis-Gundy inequality}},
	\end{split}
	\end{equation}
	where, similarly as in Lemma $4.4$ of \cite{FS99}, 
	\begin{equation*}
	\tilde{\omega}_{\varepsilon,\delta_{\varepsilon}}:=\left(\frac{2\min\left(1,\frac{\delta_{\varepsilon}}{\sqrt{\varepsilon}}\kappa_{4}\right) }{\sqrt{2\pi e}} \right)^{\kappa_{3}}+o(1).
	\end{equation*}
	
	\noindent Now, we put everything together, rescale the integral on the right of (\ref{argint}), and vary $\psi$  (over all $\psi\in \mathcal{D}\left([0,1],\R^d\right)$) such that $\psi_0=x$ and $\psi_{1}=z$. Then, by the remainder of ($\boldsymbol{D.1}$), we have
	\begin{equation}
	\begin{split}
	&\mathbb{P}(A_{\varepsilon,\delta_{\varepsilon}}^{\eta})\geqslant \exp{\left\lbrace- \frac{\scriptstyle\delta_{\varepsilon}}{\scriptstyle\varepsilon}(1+\kappa_{1}\eta)(1+\tilde{\eta})V_{1/\delta_{\varepsilon}}\left(\scriptstyle\frac{ x}{ \delta_{\varepsilon}},\frac{ z}{\delta_{\varepsilon}}, 1+\phi\right) +\frac{\scriptstyle 1}{\scriptstyle\varepsilon}\omega_{\varepsilon,\delta_{\varepsilon}}^{\eta,\tilde{\eta}}\right\rbrace }\\
	&\quad\times \exp{\Bigg\{- \frac{\scriptstyle K^{1/\kappa_{3}}}{\scriptstyle\sqrt{\varepsilon}}\frac{\sqrt{2(1+\kappa_{1}\eta)(1+\tilde{\eta})\delta_{\varepsilon} V_{1/\delta_{\varepsilon}}^1\left(\frac{x}{\delta_{\varepsilon}},\frac{z}{\delta_{\varepsilon}},\scriptstyle 1+\phi\right) + C(\scriptstyle 1+ \frac{1}{\delta_{\varepsilon}})+\omega_{\varepsilon,\delta_{\varepsilon}}^{\eta,\tilde{\eta}}}}{\left( \tilde{\omega}_{\varepsilon,\delta_{\varepsilon}}\right)^{1/\kappa_{3}}}\Bigg\}}\\
	&\quad\times \tilde{\omega}_{\varepsilon,\delta_{\varepsilon}}.
	\end{split}
	\end{equation}
	
	\noindent Thus, letting consecutively $\varepsilon$,  $\delta_{\varepsilon}$, $\eta$ and then $\tilde{\eta}$ tend to zero in that order, we get
	\[
	\liminf_{\varepsilon\to 0}\varepsilon\log\mathbb{P}\left\lbrace X_{1}^{x,\varepsilon,\delta_{\varepsilon}}\in G\right\rbrace \geqslant -\inf_{ z\in G } \mathcal{J}(z-x).
	\]
\end{proof}

\subsection{The upper bound}
\noindent \noindent  We use  $\big\{\boldsymbol{\hat{T}}_{t,s}^{\varepsilon,\delta_{\varepsilon}}: t<s\big\} $ to denote the semigroup on $C^{\infty}\left( \R^{d}\right) $ generated by  the operator $\mathcal{L}_{\varepsilon,\delta_{\varepsilon}}$ with $C_{c}^{\infty}\left( \R^{d}\right)\subset\mathcal{D}\left( \mathcal{L}_{\varepsilon,\delta_{\varepsilon}}\right) $,   and let $p^{\varepsilon,\delta_{\varepsilon}}(s-t,z,y)$ denote the heat kernel associated with the semigroup $\boldsymbol{\hat{T}}_{t,s}^{\varepsilon,\delta_{\varepsilon}}$ in the sense :
\begin{equation}
\left(\boldsymbol{\hat{T}}_{t,s}^{\varepsilon,\delta_{\varepsilon}}\psi\right)(z)=\int_{\R^d}p^{\varepsilon,\delta_{\varepsilon}}(s-t,z,y)\psi(y)dy,\ t<s,\ z\in\R^{d},\  \psi\in C_{c}^{\infty}\left( \R^{d}\right) .
\end{equation}
Thus, for any $A\in\mathcal{B}\left(\R^d\right)$ the $\sigma$-algebra of all Borel subsets of $\R^d$,
\begin{equation}\label{loi-A}
\mathbb{P}\left(X_{t}^{x,\varepsilon,\delta_{\varepsilon}}\in A\right):=\int_{A/\delta_{\varepsilon}}p^{\varepsilon,\delta_{\varepsilon}}\left( \left( \frac{\scriptstyle \sqrt{\varepsilon}}{\scriptstyle \delta_{\varepsilon}}\right)^{2}t,z,\frac{ \scriptstyle x}{\scriptstyle \delta_{\varepsilon}}\right) dz,\ t\geqslant 0.  
\end{equation} 
With our requirements, it is well know that $X^{x,\varepsilon,\delta_{\varepsilon}}$ is a strong Feller process. In fact, the conditions (\ref{c1}) and (\ref{c2}) hold true. Then let us define its symbol. Before continuing, let us set
\[
\begin{split}
\hat{\mathcal{L}}^{1}_{\varepsilon,\delta_{\varepsilon}}&:=
\frac{\scriptstyle1}{\scriptstyle2}\sum_{i,j=1}^{d}a_{ij}\left(\frac{\scriptstyle x}{\scriptstyle\delta_{\varepsilon}}\right)\frac{\scriptstyle\partial^{2}}{\scriptstyle\partial x_{i}\partial x_{j}} +\sum_{i=1}^{d}b_{i}\left(\frac{\scriptstyle x}{\scriptstyle\delta_{\varepsilon}}\right)\frac{\scriptstyle\partial  }{\scriptstyle\partial x_{i}}+\frac{\scriptstyle\delta_{\varepsilon}}{\scriptstyle\varepsilon}\sum_{i=1}^{d}c_{i}\left(\frac{\scriptstyle x}{\scriptstyle\delta_{\varepsilon}}\right) \frac{\scriptstyle\partial  }{\scriptstyle\partial x_{i}}\\
\hat{\mathcal{L}}^{2}_{\varepsilon,\delta_{\varepsilon}}&:=\int_{\R^{d}}\Bigg\{\sum_{l=1}^{d}\left( i k_{l}\left( \frac{\scriptstyle x}{\scriptstyle\delta_{\varepsilon}},y\right)  \xi_{l}+\frac{\scriptstyle\varepsilon}{\scriptstyle\delta_{\varepsilon}}k_{l}\left( \frac{\scriptstyle x}{\scriptstyle\delta_{\varepsilon}},y\right)\frac{\scriptstyle\partial  }{\scriptstyle\partial x_{l}}\right)\\
&\qquad\qquad\qquad\qquad\qquad\qquad\ +\left(e^{i\left\langle k\left( \frac{\scriptstyle x}{\scriptstyle\delta_{\varepsilon}},y\right), \xi\right\rangle} -1 \right)_{y\neq 0} \Bigg\}\nu(dy)\\
\hat{\mathcal{L}}^{3}_{\varepsilon,\delta_{\varepsilon}} &:=\frac{\scriptstyle\varepsilon}{\scriptstyle\delta_{\varepsilon}}\sum_{l,r=1}^{d}a_{l,r}\left( \frac{\scriptstyle x}{\scriptstyle\delta_{\varepsilon}}\right) \xi_{l}\frac{\scriptstyle\partial  }{\scriptstyle\partial x_{l}}+i\sum_{l=1}^{d}\left(\frac{\scriptstyle\varepsilon}{\scriptstyle\delta_{\varepsilon}}b_{l}+c_{l}\right)\left( \frac{\scriptstyle x}{\scriptstyle\delta_{\varepsilon}}\right) \xi_{l}-\frac{\scriptstyle1}{\scriptstyle2}\sum_{l,r=1}^{d} \xi_{l}^{*}a_{l,r}\left( \frac{\scriptstyle x}{\scriptstyle\delta_{\varepsilon}}\right) \xi_{r}.\\
\end{split}
\]
Next, we have
\begin{equation}
\begin{split}
-q^{\varepsilon,\delta_{\varepsilon}}\left(\frac{\scriptstyle x}{\scriptstyle\delta_{\varepsilon}},\xi\right)&= \varepsilon\exp\left( -i\Big\langle \frac{\scriptstyle x}{\scriptstyle\delta_{\varepsilon}},\xi\Big\rangle \right) \mathcal{L}_{\varepsilon,\delta_{\varepsilon}} \exp\left( i\Big\langle \frac{\scriptstyle x}{\scriptstyle\delta_{\varepsilon}},\xi\Big\rangle\right)  \\
&=\left(\frac{\scriptstyle\varepsilon}{\scriptstyle\delta_{\varepsilon}}\right)^{2}\hat{\mathcal{L}}^{1}_{\varepsilon,\delta_{\varepsilon}}+\hat{\mathcal{L}}^{2}_{\varepsilon,\delta_{\varepsilon}}+\hat{\mathcal{L}}^{3}_{\varepsilon,\delta_{\varepsilon}}
\end{split}
\end{equation}
Heuristically if $\dfrac{\scriptstyle x}{\scriptstyle\delta_{\varepsilon}}\longrightarrow z$ when $\varepsilon\longrightarrow 0$, it can be seen that this operator  converges to $-q$ defined as :
\begin{equation}
\begin{split}
-q(z,\xi)=&-\frac{\scriptstyle1}{\scriptstyle2}\sum_{l,r=1}^{d} \xi_{l}^{*}a_{l,r}\left( z\right)\xi_{r}+i\sum_{r=1}^{d}\left\lbrace c_{i}\left( z\right) \xi_{r}+\int_{\R^{d}}k_{r}\left( z,y\right)\xi_{r}\nu(dy)\right\rbrace\\
&+\int_{\R^d\backslash\{0\}}\left(e^{i\left\langle k\left( z,y\right), \xi\right\rangle} -1 \right)\nu(dy) .
\end{split}
\end{equation}

\noindent Now we have the following \textit{upper bound} in space.
\begin{proposition}\label{prop2}
	Suppose the assumptions ($ \boldsymbol{H.1}$) to ($\boldsymbol{H.2}$) hold true.	for each closed subset $\ F\subseteq \R^d$
	\[
	\limsup_{\varepsilon\to 0}\varepsilon\log\mathbb{P}\left\lbrace X_{1}^{x,\varepsilon,\delta_{\varepsilon}}\in F\right\rbrace \leqslant -\inf_{z\in F} \mathcal{J}(z-x).
	\]
\end{proposition}
\begin{proof}
	Before proceeding, let $\hat{Q}_{1}$ be the quadratic form defined as $\hat{Q}_{1}(v):=\left\langle v,av\right\rangle $ and let $\hat{Q}_{1}^{*}$ be the conjugate quadratic form of $\hat{Q}_{1}$ defined as :
	\[
	\hat{Q}_{1}^{*}(v):=\sup_{t\in\R^{d}}\biggl\{2\left\langle t,v\right\rangle-\hat{Q}_{1}(t)\biggr\}.
	\]
	It is well know that if the inverse of the matrix $a$ exists, then $$\hat{Q}_{1}^{*}(v):=\left\langle v,a^{-1}v\right\rangle .$$
	Next, fix $\theta\in\R^d$ and let $\displaystyle\hat{Q}^{*}_{2}(v,\theta)=\sup_{t\in\R^d}\displaystyle\Bigl\{\left\langle t,v\right\rangle-\Big[\exp\left( \left\langle \theta,t\right\rangle \right)-1\Big]\Bigr\}$. Then, it is easy to see that
	\[
	\hat{Q}^{*}_{2}(v,\theta)=\hat{Q}_{2}\left(\left\| v\right\|\times \left\| \theta\right\|^{-1} \right),\ \textrm{ with }\ \hat{Q}_{2}(x):=x\log x-x+1 .
	\]
	
	\noindent Now, fix $x,z\in\R^d$ and $t>0$, Lemma \ref{hwc} tells us
	\begin{equation*}
	\begin{split}
	\sup_{x,z\in\R^d}p^{\varepsilon,\delta_{\varepsilon}}\left(t,x,z\right) &\leqslant\int\exp\left( -\frac{\scriptstyle 1}{\scriptstyle16}\inf_{\substack{\psi\in \mathcal{D}\left(\left[0,t\right];\R^d\right)\\\psi_0=x,\ \psi_t=z}}\int_{0}^{t}\textrm{Re}q^{\varepsilon,\delta_{\varepsilon}}\left( \frac{\scriptstyle\psi_{s}}{\scriptstyle\delta_{\varepsilon}},\xi\right)ds \right)d\xi \\
	&\leqslant\int\exp\left( \frac{\scriptstyle 1}{\scriptstyle2}\inf_{\substack{\psi\in \mathcal{D}\left(\left[0,t\right];\ \R^d\right)\\\psi_0=x,\ \psi_t=z}}\int_{0}^{t}\big\langle \xi,a\left( \frac{\scriptstyle\psi_{s}}{\scriptstyle\delta_{\varepsilon}}\right) \xi\big\rangle ds\right. \\
	+\inf_{\substack{\psi\in \mathcal{D}\left(\left[0,t\right];\ \R^d\right)\\\psi_0=x,\ \psi_t=z}}& \left. \int_{0}^{t}\int_{\R^d}\textrm{Re}\Biggl\{1-e^{\Big(\displaystyle i\big\langle k\left( \frac{\scriptstyle\psi_{s}}{\scriptstyle\delta_{\varepsilon}},y\right),\xi\big\rangle \Big)}\Biggr\}_{y\neq 0}dsdy \right)  +o(1).
	\end{split}
	\end{equation*}
	Thus
	\begin{equation*}
	\begin{split}
	&\sup_{x,z\in\R^d}p^{\varepsilon,\delta_{\varepsilon}}\left(t,x,z\right) \\
	&\leqslant\exp\left( -\frac{\scriptstyle 1}{\scriptstyle2}\inf_{\substack{\psi\in \mathcal{D}\left(\left[0,t\right];\ \R^d\right)\\\psi_0=x,\ \psi_t=z}}\int_{0}^{t}\hat{Q}_{1}^{*}\left(\dot{\psi}_{s}-\frac{\scriptstyle\delta_{\varepsilon}}{\scriptstyle\varepsilon}\left\lbrace c+\int_{\R^d}k(.,y)\nu(dy)\right\rbrace \left( \frac{\scriptstyle\psi_{s}}{\scriptstyle\delta_{\varepsilon}}\right) \right)ds\right) \\
	&\times\exp\left(\inf_{\substack{\psi\in \mathcal{D}\left(\left[0,t\right];\ \R^d\right)\\\psi_0=x,\ \psi_t=z}}\sup_{\xi\in\R^d}\int_{0}^{t}\Big\langle \dot{\psi}_s-\frac{\scriptstyle\delta_{\varepsilon}}{\scriptstyle\varepsilon}\left\lbrace c+\int_{\R^d}k(.,y)\nu(dy)\right\rbrace \left( \frac{\scriptstyle\psi_{s}}{\scriptstyle\delta_{\varepsilon}}\right),\xi \Big\rangle ds \right)\\
	&\times\int\exp\left( \inf_{\substack{\psi\in \mathcal{D}\left(\left[0,t\right];\ \R^d\right)\\\psi_0=x,\ \psi_t=z}}\int_{0}^{t}\int_{\R^d}\textrm{Re}\Biggl\{1-e^{\Big(\displaystyle i\big\langle k\left( \frac{\scriptstyle\psi_{s}}{\scriptstyle\delta_{\varepsilon}},y\right),\xi\big\rangle \Big)}\Biggr\}_{y\neq 0}dsdy\right)d\xi\\
	&+o(1).
	\end{split}
	\end{equation*}
	Thereupon, we have
	\begin{equation*}
	\begin{split}
	&\sup_{x,z\in\R^d}p^{\varepsilon,\delta_{\varepsilon}}\left(t,x,z\right) \\
	&\leqslant\exp\left( -\frac{\scriptstyle 1}{\scriptstyle2}\inf_{\substack{\psi\in \mathcal{D}\left(\left[0,t\right];\ \R^d\right)\\\psi_0=x,\ \psi_t=z}}\int_{0}^{t}\hat{Q}_{1}^{*}\left( \dot{\psi}_{s}-\frac{\scriptstyle\delta_{\varepsilon}}{\scriptstyle\varepsilon}\left\lbrace c+\int_{\R^d}k(.,y)\nu(dy)\right\rbrace \left( \frac{\scriptstyle\psi_{s}}{\scriptstyle\delta_{\varepsilon}}\right) \right)ds\right) \\
	&\times K\exp\Bigg( -\inf_{\substack{\psi\in \mathcal{D}\left(\left[0,t\right];\ \R^d\right)\\\psi_0=x,\ \psi_t=z}}\Biggl\{\int_{0}^{t}\hat{Q}^{*}_{2}\left( \dot{\psi}_{s},\int_{\R^{d}}k\left(\frac{\scriptstyle\psi_{s}}{\scriptstyle\delta_{\varepsilon}},y\right)dy\right)ds \\
	&\quad\qquad\qquad+\left. \frac{\scriptstyle\delta_{\varepsilon}}{\scriptstyle\varepsilon}\sup_{\xi\in\R^d}\int_{0}^{t}\Big\langle \left\lbrace c+\int_{\R^d}k(.,y)\nu(dy)\right\rbrace \left( \frac{\scriptstyle\psi_{s}}{\scriptstyle\delta_{\varepsilon}}\right),\xi \Big\rangle ds\Bigg]\Biggr\}\right) +o(1).
	\end{split}
	\end{equation*}
	This follows
	\begin{equation}
	\begin{split}\label{mac}
	&\sup_{x,z\in\R^d}p^{\varepsilon,\delta_{\varepsilon}}\left(t,\frac{\scriptstyle x}{\scriptstyle\delta_{\varepsilon}},\frac{\scriptstyle z}{\scriptstyle\delta_{\varepsilon}}\right) \\
	&\leqslant\exp\left( -\frac{\scriptstyle 1}{\scriptstyle2}\inf_{\substack{\psi\in \mathcal{D}\left(\left[0,t\right];\ \R^d\right)\\\psi_0=\frac{x}{\delta_{\varepsilon}},\ \psi_t=\frac{z}{\delta_{\varepsilon}}}}\int_{0}^{t}\hat{Q}_{1}^{*}\left(\frac{\scriptstyle\dot{\psi}_{s}}{\scriptstyle\delta_{\varepsilon}}-\frac{\scriptstyle\delta_{\varepsilon}}{\scriptstyle\varepsilon}\left\lbrace c+\int_{\R^d}k(.,y)\nu(dy)\right\rbrace \left( \frac{\scriptstyle\psi_{s}}{\scriptstyle\delta_{\varepsilon}}\right) \right)ds\right) \\
	&\times K\exp\Bigg( -\inf_{\substack{\psi\in \mathcal{D}\left(\left[0,t\right];\ \R^d\right)\\\psi_0=\frac{x}{\delta_{\varepsilon}},\ \psi_t=\frac{z}{\delta_{\varepsilon}}}}\Biggl\{\int_{0}^{t}\hat{Q}^{*}_{2}\left(\frac{\scriptstyle\dot{\psi}_{s}}{\scriptstyle\delta_{\varepsilon}},\int_{\R^{d}}k\left(\frac{\scriptstyle\psi_{s}}{\scriptstyle\delta_{\varepsilon}},y\right)dy\right)ds \\
	&\quad\qquad\qquad+\left. \frac{\scriptstyle\delta_{\varepsilon}}{\scriptstyle\varepsilon}\sup_{\xi\in\R^d}\int_{0}^{t}\Big\langle \left\lbrace c+\int_{\R^d}k(.,y)\nu(dy)\right\rbrace \left( \frac{\scriptstyle\psi_{s}}{\scriptstyle\delta_{\varepsilon}}\right),\xi \Big\rangle ds\Bigg]\Biggr\}\right) +o(1).
	\end{split}
	\end{equation}
	From (\ref{loi-A}) we have
	\begin{equation}
	\begin{split}
	\mathbb{P}\left( X_{1}^{x,\varepsilon,\delta_{\varepsilon}}\in F\right) &=\int_{ F/\delta_{\varepsilon}}p^{\varepsilon,\delta_{\varepsilon}}\left(\left( \sqrt{\varepsilon}/\delta_{\varepsilon}\right)^{2},\frac{\scriptstyle  x}{\scriptstyle \delta_{\varepsilon}},z\right)dz\\
	&=\underbrace{\delta_{\varepsilon}^{-d}\int_{F}p^{\varepsilon,\delta_{\varepsilon}}\left(\left( \sqrt{\varepsilon}/\delta_{\varepsilon}\right)^{2},\frac{\scriptstyle x}{\scriptstyle \delta_{\varepsilon}},\frac{\scriptstyle  z}{\scriptstyle \delta_{\varepsilon}}\right)dz}_{\textrm{by scaling property}}.
	\end{split}
	\end{equation}
	By (\ref{mac}), we deduct
	\[
	\begin{split}
	&\varepsilon\log\mathbb{P}\left( X^{x,\varepsilon,\delta_{\varepsilon}}_{1}\in F\right)\\
	&\leqslant-\frac{\scriptstyle \varepsilon}{\scriptstyle2}\inf_{\substack{\psi\in \mathcal{D}\left(\left[0,\frac{\varepsilon}{\delta_{\varepsilon}^{2}}\right];\R^d\right)\\\psi_0=\frac{x}{\delta_{\varepsilon}},\ \psi_{\frac{\varepsilon}{\delta_{\varepsilon}^{2}}}=\frac{z}{\delta_{\varepsilon}}}}\int_{0}^{\frac{\varepsilon}{\delta_{\varepsilon}^{2}}}\hat{Q}_{1}^{*}\left(\frac{\scriptstyle\dot{\psi}_{s}}{\scriptstyle\delta_{\varepsilon}}-\frac{\scriptstyle\delta_{\varepsilon}}{\scriptstyle\varepsilon}\left\lbrace c+\int_{\R^d}k(.,y)\nu(dy)\right\rbrace \left(\frac{\scriptstyle\psi_{s}}{\scriptstyle\delta_{\varepsilon}}\right) \right)ds\\
	&\quad-\varepsilon\inf_{\substack{\psi\in \mathcal{D}\left(\left[0,\frac{\varepsilon}{\delta_{\varepsilon}^{2}}\right];\R^d\right)\\\psi_0=\frac{x}{\delta_{\varepsilon}},\ \psi_{\frac{\varepsilon}{\delta_{\varepsilon}^{2}}}=\frac{z}{\delta_{\varepsilon}}}}\Biggl\{\int_{0}^{\frac{\varepsilon}{\delta_{\varepsilon}^{2}}}\hat{Q}^{*}_{2}\left(\frac{\scriptstyle\dot{\psi}_{s}}{\scriptstyle\delta_{\varepsilon}},\int_{\R^{d}}k\left(\frac{\scriptstyle\psi_{s}}{\scriptstyle\delta_{\varepsilon}},y\right)dy\right)ds \\
	&\quad\qquad\qquad+\frac{\scriptstyle\delta_{\varepsilon}}{\scriptstyle\varepsilon}\sup_{\xi\in\R^d}\int_{0}^{\frac{\varepsilon}{\delta_{\varepsilon}^{2}}}\Big\langle \left\lbrace c+\int_{\R^d}k(.,y)\nu(dy)\right\rbrace \left( \frac{\scriptstyle\psi_{s}}{\scriptstyle\delta_{\varepsilon}}\right),\xi \Big\rangle ds\Bigg]\Biggr\} +o(1)\\
	&\leqslant-\frac{\scriptstyle\delta_{\varepsilon}}{\scriptstyle2}\inf_{\substack{\psi\in \mathcal{D}\left(\left[0,\frac{1}{\delta_{\varepsilon}}\right];\R^d\right)\\\psi_0=\frac{x}{\delta_{\varepsilon}},\ \psi_{\frac{1}{\delta_{\varepsilon}}}=\frac{z}{\delta_{\varepsilon}}}}\int_{0}^{\frac{1}{\delta_{\varepsilon}}}\hat{Q}_{1}^{*}\left(\dot{\psi}_{s}-\left\lbrace c+\int_{\R^d}k(.,y)\nu(dy)\right\rbrace \left(\psi_{s}\right) \right)ds\\
	&\quad-\inf_{\substack{\psi\in \mathcal{D}\left(\left[0,\frac{\varepsilon}{\delta_{\varepsilon}}\right];\R^d\right)\\\psi_0=\frac{x}{\delta_{\varepsilon}},\ \psi_{\frac{1}{\delta_{\varepsilon}}}=\frac{z}{\delta_{\varepsilon}}}}\Biggl\{\delta_{\varepsilon}\int_{0}^{\frac{1}{\delta_{\varepsilon}}}\hat{Q}^{*}_{2}\left( \dot{\psi}_{s},\int_{\R^{d}}k\left(\psi_{s},y\right)dy\right)ds \\
	&\quad\qquad\qquad+\varepsilon\sup_{\xi\in\R^d}\int_{0}^{\frac{1}{\delta_{\varepsilon}}}\Big\langle \left\lbrace c+\int_{\R^d}k(.,y)\nu(dy)\right\rbrace \left(\psi_{s}\right),\xi \Big\rangle ds\Bigg]\Biggr\} +o(1).
	\end{split}
	\]
	Therefore, the claim follows, \textit{i.e}:
	\[
	\begin{split}
	\lim_{\varepsilon\to 0}\varepsilon\log\mathbb{P}\left( X^{x,\varepsilon,\delta_{\varepsilon}}_{1}\in F\right)&\leqslant-\lim_{\varepsilon\to 0}\delta_{\varepsilon}V^{1}_{1/\delta_{\varepsilon}}(x,z)-\liminf_{\varepsilon\to 0}\delta_{\varepsilon} V^{2}_{1/\delta_{\varepsilon}}(1+\phi)\\
	&\leqslant-\inf_{z\inf F}\mathcal{J}(z-x),
	\end{split}
	\]
	where $\displaystyle\phi+1:=\left\| \dot{\psi}\right\| \times\left(\int_{0}^{.}\int_{\R^d}\left\| k(.,y)\right\| \nu(dy) ds\right) ^{-1}$.
\end{proof}
\subsection{Tightness}
\noindent Let $\mathcal{D}^{\lambda}\left([0,T],\R^d \right)$ denotes the space of H\"{o}lder-c\`{a}ld\`{a}g functions of exponent $\lambda$ and $\left\|.\right\|_{\mathcal{D}^{\lambda}\left([0,T],\R^d\right)}$ its corresponding norm. We need the following remark with projective limit approach (see, for example \cite{DZ93}) to guess at the path-space large deviations principle.  
\begin{rem}
	for any fixed $T>0$, $x\in\R^d$ and $\lambda\in\left(0,1/2\right)$,
	\[
	\lim_{L\to +\infty}\limsup_{\varepsilon\to 0}\varepsilon\log\mathbb{P}\left\lbrace \left\|X^{x,\varepsilon,\delta_{\varepsilon}} \right\|_{\mathcal{D}^{\lambda}\left([0,T],\boldsymbol{\overline{D}} \right)}\geqslant L\right\rbrace=-\infty. 
	\]
\end{rem}


\vskip.5cm \noindent{\bf Acknowledgement(s) :} We would like to thank the referee(s) for his comments and suggestions on the manuscript.

\medskip

%
%
%

\vskip.5cm

\end{document}